\documentclass[11pt, leqno]{amsart}
\usepackage{amsfonts,amssymb,amscd,amsmath,amsthm}
\usepackage{tikz}
\usepackage[nice]{nicefrac}
\usepackage{lmodern}
\usepackage{mathrsfs}
\usepackage{qsymbols}
\usepackage{mathtools}
\mathtoolsset{showonlyrefs}
\usepackage{dsfont}
\usepackage{graphicx}
\usepackage{latexsym}
\usepackage{chngcntr}
\usepackage{paralist}
\usepackage[parfill]{parskip}
\usepackage{bm}
\usepackage{tikz-cd}
\usepackage{csquotes}
\usepackage[plainpages=false,pdfpagelabels,backref=page,
citecolor=red]{hyperref}
\usepackage{xcolor}
\hypersetup{
	colorlinks,
	linkcolor={cyan!90!black},
	citecolor={magenta},
	urlcolor={green!40!black}}

\let\oldproofname=\proofname
\renewcommand{\proofname}{\rm\bf{\oldproofname}}
\usepackage[shortlabels]{enumitem}

\newtheorem{theorem}{Theorem}[section]
\newtheorem{lemma}[theorem]{Lemma}
\newtheorem{proposition}[theorem]{Proposition}
\newtheorem{corollary}[theorem]{Corollary}
\newtheorem{assumption}[theorem]{Assumption}
\theoremstyle{definition}
\newtheorem{definition}[theorem]{Definition}
\newtheorem{remark}[theorem]{Remark}

\numberwithin{equation}{section}

\DeclareMathOperator{\e}{e}

\DeclareMathOperator{\Div}{div}

\DeclareMathOperator{\cc}{c}

\newcommand{\C}{\mathbb{C}}
\newcommand{\R}{\mathbb{R}}
\newcommand{\N}{\mathbb{N}}

\renewcommand{\H}{\mathrm{H}}
\newcommand{\W}{\mathrm{W}}

\DeclareRobustCommand{\Wdot}{\dot{\W}\protect{\vphantom{W}}}

\renewcommand{\L}{\mathrm{L}}

\newcommand{\LL}{\mathcal{L}}
\newcommand{\D}{\mathsf{D}}

\renewcommand{\d}{\mathrm{d}}

\DeclareMathOperator{\essinf}{essinf}
\newcommand{\les}{\lesssim}

\renewcommand{\S}{\mathrm{S}}

\DeclareMathOperator{\re}{Re}

\newcommand{\1}{\boldsymbol{1}}

\newcommand{\rC}{\mathrm{C}}

\newcommand{\cJ}{\mathcal{J}}

\newcommand{\cN}{\mathcal{N}}

\newcommand{\up}{\ldots}
\newcommand{\sub}{\subseteq}
\newcommand{\smooth}[1][]{\rC_{\cc}^{\infty}(#1)}

\def\XXint#1#2#3{{\setbox0=\hbox{$#1{#2#3}{%
				\int}$ }
		\vcenter{\hbox{$#2#3$ }}\kern-.6\wd0}}
\title{Explicit improvements for $\L^p$-estimates related to elliptic systems}
\author{Tim B\"ohnlein}
\author{Moritz Egert}
\address{Fachbereich Mathematik, Technische Universit\"at Darmstadt, Schlossgartenstr. 7, 64289 Darmstadt, Germany}
\email{boehnlein@mathematik.tu-darmstadt.de}
\email{egert@mathematik.tu-darmstadt.de}
\date{\today}
\dedicatory{}

\begin{document}
\begin{abstract}
	We give a simple argument to obtain $\L^p$-boundedness for heat semigroups associated to uniformly strongly elliptic systems on $\R^d$ by using Stein interpolation between Gaussian estimates and hypercontractivity. Our results give $p$ explicitly in terms of ellipticity. It is optimal at the endpoint $p=\infty$. We also obtain $\L^p$-estimates for the gradient of the semigroup, where $p>2$ depends on ellipticity but not on dimension.
\end{abstract}
\maketitle

\section{Introduction}

In dimension $d \geq 3$ we consider uniformly strongly elliptic systems on $\R^d$ of $N\geq1$ equations in divergence form
\begin{equation*}
 (Lu)^{\alpha} = - (\Div (A \nabla u))^{\alpha}  = - \sum_{i,j =1}^d \sum_{\beta =1}^N \partial_i A^{\alpha, \beta}_{i,j} \partial_j u^{\beta} \qquad (\alpha =1, \up, N)
\end{equation*}
with bounded, measurable and complex coefficients, see Section~\ref{Section: Strongly elliptic systems} for precise definitions. This gives rise to a contraction semigroup $(\e^{-tL})_{t > 0}$ in $(\L^2)^N \coloneqq \L^2(\R^d; \C^N)$. Surprisingly little is know about explicit $\L^p$-estimates when no further regularity on the coefficients is imposed. For systems with minimally smooth coefficients we refer e.g.\@ to \cite{Dong-Kim}. More precisely, consider the set
\begin{equation*}
    \cJ(L) \coloneqq \left\{ p \in (1, \infty) : \e^{-t L} \; \text{is bounded in $\L^p$, uniformly for $t > 0$} \right\}.
\end{equation*}
By complex interpolation, it is an interval around $2$, the endpoints of which are often denoted by $p_{\pm}(L)$. All of our results will be stable under taking adjoints. Since $p_-(L) = (p_+(L^*))'$, we shall concentrate on the upper endpoint $p_+(L)$. It is known that $p_+(L) > 2^*$, where $2^* \coloneqq \nicefrac{2d}{(d-2)}$ is the Sobolev conjugate of $2$, and that the improvement $p_+(L) - 2^*$ can be arbitrarily small even when $N=1$~\cite[Sec.~2.2]{H_M_McI}. What seems to be missing though, are explicit lower bounds for $p_+(L)$ in terms of the data of $L$, such as ellipticity constants and dimensions, in particular when the improvement is expected to be large or even covers $p_+(L)=\infty$. Indeed, all results for systems (that we are aware of) are perturbative from the general $\L^2$-theory and provide small, non-quantifiable improvements~\cite{Tolksdorf, Auscher_Heat-Kernel, Auscher_Riesz_Trafo}. In contrast, we proceed by interpolation from the $\L^\infty$-theory for special systems described further below. Our results are new also for elliptic equations ($N=1$).

The number $p_+(L)$ is related to the optimal ranges of various $\L^p$-estimates for $L$, such as Riesz transforms, boundary value problems and functional calculus, see the introduction of \cite{Auscher-Egert} for a comprehensive account on the literature. Thus, improving lower bounds for $p_+(L)$, as we shall do here, leads to automatic improvements in all these topics.

All of our results are perturbative from the diagonal Laplacian system corresponding to $A=\1_{(\C^N)^d} \coloneqq (\delta_{\alpha,\beta}\delta_{i,j})^{\alpha, \beta}_{i,j}$ but not necessarily on a small scale. This is in the nature of things, because every strongly elliptic $A$ is an $\L^{\infty}$-perturbation of $\1_{(\C^N)^d}$ of size smaller than $1$ up to normalization:
\begin{equation*}
    \d(A) \coloneqq \min_{t\geq 0} \Vert \1_{(\C^N)^d} - t A \Vert_{\L^{\infty}(\R^d; \LL((\C^N)^d))} < 1.
\end{equation*}
The `distance' $\d(A)$ is a well-known measure of ellipticity~\cite{Koshelev}. It can be bounded from above and below in terms of the usual ellipticity constants and when $A=A^*$, there is an easy formula (Lemma~\ref{Strongly elliptic systems: Lemma: q(A) vs. d(A)}). The dimensional constant
\begin{equation}
\label{eq: delta(d)}
    \delta(d) \coloneqq \left(1 + \frac{(d-2)^2}{d-1} \right)^{-\frac{1}{2}}
\end{equation}
will play an important role in this paper. Our main result is as follows.

\begin{theorem}
\label{Introduction: Theorem: Application of the abstract interpolation result}
The following three statements hold.
\begin{enumerate}
    \item If $\d(A) \geq \delta(d)$, then (with $\nicefrac{2^*}{0} \coloneqq \infty$)
    \begin{equation*}
    p_+(L) \geq \frac{2^*}{1 - \frac{\ln(\d(A))}{\ln(\delta(d))}}.
    \end{equation*}
    \item If $\d(A) < \delta(d)$, then  $p_+(L) = \infty$.
    \item The result in (ii) is optimal in the sense that for each $\varepsilon >0$ there is some $A_{\varepsilon}$ with $\d(A_{\varepsilon}) \leq \delta(d) + \varepsilon$ and $p_+(L_{\varepsilon}) < \infty$.
\end{enumerate}
\end{theorem}

The dimensional constants in Theorem~\ref{Introduction: Theorem: Application of the abstract interpolation result} are quite large in small dimensions and we collect some values in Figure~\ref{fig: OPR}.
\begin{figure}[!h]
\begin{center}
\begin{tabular}[h]{l|c|c|r}
$d$  & $\delta(d)$ & $-\ln(\delta(d))^{-1}$ & $\varrho(A)$\\
\hline
3  & $0.8165$  & $4.9326$ & $0.1010$\\
4  & $0.6547$  & $2.3604$ & $0.2087$\\
5  & $0.5547$  & $1.6968$ & $0.2864$\\
6  & $0.4880$ & $1.3936$ & $0.3441$
\end{tabular}
\end{center}
\caption{Approximate constants in Theorem~\ref{Introduction: Theorem: Application of the abstract interpolation result} in small dimensions. The third column contains the ellipticity ratio $\varrho(A) = \nicefrac{\lambda(A)}{\Lambda(A)} \in (0,1]$ that is sufficient for having $\d(A) = \delta(d)$
in the special case $A=A^*$, see Lemma~\ref{Strongly elliptic systems: Lemma: q(A) vs. d(A)}.}
\label{fig: OPR}
\end{figure}

Part (ii) is proved in Section~\ref{Section: OPR whole space} by combining results of Koshelev~\cite{Koshelev}, see also \cite{Leonardi}, with a characterization of Gaussian estimates and Hölder regularity of the kernel associated with $\e^{-tL}$ due to Auscher--Tchamitchian~\cite{Auscher_Tchamitchian_Kato}. In fact, in Theorem~\ref{OPR whole space: Theorem: H(mu) implies G(mu)} we shall not only prove that $\d(A) < \delta(d)$ implies $p_+(L) = \infty$, but that $\e^{-tL}$ has a H\"older regular integral kernel with Gaussian decay. Since $\delta(d) > \nicefrac{1}{\sqrt{d}}$, this disproves the conjecture in \cite[Chap.~1, Sec.~1.4.6]{Auscher_Tchamitchian_Kato} that the best possible perturbation result would be $\d(A) \in O(d^{-1})$. The optimality statement in (iii) is almost classical, see Proposition~\ref{OPR whole space: Proposition: p+ optimal}.

Part (i) is proved in Section~\ref{Section: A priori estimates: Interpolation approach}. The idea is to rewrite $\d(A) < 1$ as $A = \tau (\1_{(\C^N)^d} - B)$, where $\| B \|_{\infty} =\d(A)$ and $\tau > 0$. We embed $A$ as $A_1$ into an analytic family of elliptic matrices given by
\begin{equation*}
    A_z \coloneqq \tau (\1_{(\C^N)^d} - z B),
\end{equation*}
where $r \leq |z|\leq R$ with $0 < r < 1 < R$. Then, in the spirit of Stein interpolation, we estimate $\e^{-tL} = \e^{-t L_1}$ by using the generic information $p_+(L_z) \geq 2^*$ on the outer circle $|z| = R$ and $p_+(L_z) = \infty$ on the inner circle $|z|=r$ provided $r$ is small. This gives a lower bound for $p_+(L_1)$ that becomes the larger, the closer $z=1$ is to the inner circle and the farther away it is from the outer one. Thus, the best bound is achieved when $r, R$ are the largest possible and the optimal choice for $r$ comes from (ii). We believe that this simple analytic perturbation argument is of independent interest and has multiple applications to other types of $\L^p$-estimates for divergence form operators.

Writing Theorem~\ref{Introduction: Theorem: Application of the abstract interpolation result} (i) as $\nicefrac{1}{2^*} - \nicefrac{1}{p_+(L)} \geq \nicefrac{\ln(\d(A))}{2^* \ln (\delta(d))} \eqqcolon \varepsilon(d, \d(A))$, we see that $\varepsilon(d,\d(A)) \to 0$ as $d \to \infty$. Inspired by Stein's result~\cite{Stein_d_independence_of_RT_norm} on dimensionless bounds for the Riesz transform, we ask whether an improvement can be given independently of $d$. To this end, it will be advantageous to consider
\begin{equation*}
    \cN(L) \coloneqq \left\{ p \in (1, \infty) : \sqrt{t} \nabla \e^{-t L} \; \text{is bounded in $\L^p$, uniformly for $t > 0$} \right\},
\end{equation*}
instead of $\cJ(L)$. It is again an interval around $2$. The left and right endpoints of $\cN(L)$ are denoted by $q_{\pm}(L)$ and it is a fact that $q_-(L) = p_-(L)$ and $p_+(L) \geq q_+(L)^*$, the Sobolev conjugate of $q_+(L)$~\cite[Sec.\@ 3.4]{Auscher_Riesz_Trafo}. It follows that the improvement $q_+(L)-2$ can be arbitrarily small. In the next result, proved in Section~\ref{Section: dimensionless}, we improve $q_+(L)$ in terms of $\d(A)$ alone. Writing the conclusion as $\nicefrac{1}{2} - \nicefrac{1}{q_+(L)} \geq \varepsilon(\d(A))$ gives the dimensionless improvement $\nicefrac{1}{2^*} - \nicefrac{1}{p_+(L)} \geq \varepsilon(\d(A))$.

\begin{theorem}
\label{Introduction: Theorem: q+ improvement}
It holds
\begin{equation*}
    q_+(L) \geq
    \begin{cases}
    \frac{2}{1 + \frac{\sigma -1}{\sigma^2} \ln(\d(A))} &\quad \text{if} \quad \frac{1}{4 (\sigma -1)^2} \leq \d(A), \\
    \frac{1}{2 \sqrt{\d(A)}} +1  &\quad \text{if} \quad \d(A) \leq \frac{1}{4 (\sigma -1)^2},
    \end{cases}
\end{equation*}
where $\sigma \approx 5.69061$ is the unique real solution to
\begin{equation*}
    \ln (2 \sigma -2 ) = \frac{\sigma(\sigma-2)}{2(\sigma-1)}.
\end{equation*}
\end{theorem}

For curiosity, let us mention that the first bound in Theorem~\ref{Introduction: Theorem: q+ improvement} produces a larger improvement for $p_+(L)$ compared to Theorem~\ref{Introduction: Theorem: Application of the abstract interpolation result} (i) in dimension $d \geq 922100$.

It would be interesting to know to what extent our results can be extended to more general domains and boundary conditions. In case of Theorem \ref{Introduction: Theorem: Application of the abstract interpolation result} we provide an extension to bounded $\rC^1$-domains with Dirichlet boundary conditions in Section~\ref{Section: OPR for bounded domains}.

\textbf{Implicit constants.} We write $\mathrm{X} \les_a \mathrm{Y}$, if $\mathrm{X} \leq c \mathrm{Y}$ for some $c = c(a) > 0$.

\textbf{Acknowledgment} The authors would like to thank Salvatore Leonardi for helping us with the literature underlying Section~\ref{Section: OPR whole space}.

\section{Uniformly strongly elliptic systems} \label{Section: Strongly elliptic systems}

Let $d \geq 3$, $N \geq 1$ and $A \colon \R^d \to \LL((\C^N)^d)$ be measurable. We assume that $A$ is \textbf{uniformly strongly elliptic}, that is
\begin{equation*}
     \lambda(A)\coloneqq \underset{x \in \R^d}{\essinf} \min_{|\xi| =1} \re ( A(x) \xi \cdot \overline{\xi} ) > 0 \quad \& \quad \Lambda(A) \coloneqq \Vert A \Vert_{\L^{\infty}(\R^d; \LL((\C^N)^d))} < \infty.
\end{equation*}
Let $L = - \Div (A \nabla \cdot)$ be realized as an m-accretive operator in $(\L^2)^N$ via the sesquilinear form
\begin{equation*}
    a(u,v) \coloneqq \int_{\R^d} A \nabla u \cdot \overline{\nabla v} \, \d x = \sum_{i,j =1}^d \sum_{\alpha, \beta =1}^N \int_{\R^d} A^{\alpha, \beta}_{i,j} \partial_j u^{\beta} \overline{\partial_i v^{\alpha}} \, \d x \qquad (u,v \in (\W^{1,2})^N)
\end{equation*}
and let $(\e^{-tL})_{t\geq 0}$ be the associated contraction semigroup, see \cite[Chap.\@ 6]{Kato}. We use the \textbf{distance function}
\begin{equation} \label{eq: Distance function}
    \d(A) \coloneqq \min_{t \geq 0} \Vert \1_{(\C^N)^d} - t A \Vert_{\L^{\infty}(\R^d; \LL((\C^N)^d))},
\end{equation}
to measure ellipticity. By compactness, the minimum is attained in some $t^* \geq 0$. Let us verify that $t^*>0$ and that $\d(A)$ is an appropriate quantity to measure ellipticity.

\begin{lemma} \label{Strongly elliptic systems: Lemma: q(A) vs. d(A)}
If $\varrho(A) \coloneqq \nicefrac{\lambda(A)}{\Lambda(A)}$ denotes the ellipticity quotient of $A$, then
\begin{equation}
    \frac{1 -\varrho(A)}{1 + \varrho(A)} \leq \d(A) \leq \sqrt{1 - \varrho(A)^2}.
\end{equation}
Furthermore, if $A = A^*$, then the first inequality becomes an equality.
\end{lemma}

\begin{proof}
We have for all $t \geq 0$, each normalized $\xi \in (\C^N)^d$ and almost every $x \in \R^d$ that
\begin{equation*}
    |\xi - t A(x) \xi|^2 = 1 - 2 t \re (A(x) \xi \, | \, \xi ) + t^2 |A(x) \xi|^2 \leq 1 -2 t \lambda(A) + t^2 \Lambda(A)^2.
\end{equation*}
We choose $t \coloneqq \nicefrac{\lambda(A)}{\Lambda(A)^2}$ to get the upper bound for $\d(A)$.

Now, fix $t^* > 0$ such that $\Vert \1_{(\C^N)^d} - t^* A \Vert_{\infty} = \d(A)$. Then
\begin{equation*}
    \re (t^* A(x) \xi \, | \, \xi) = 1 - \re ((\1_{(\C^N)^d} - t^* A(x)) \xi \, | \, \xi ) \geq 1 - \d(A).
\end{equation*}
Hence, $t^* \lambda(A) \geq 1 - \d(A) >0$ and by the triangle inequality $t^* \Lambda(A) \leq 1 + \d(A)$. Rearranging gives the lower bound for $\d(A)$.

For the second claim we take $t \coloneqq \nicefrac{2}{(\Lambda(A) + \lambda(A))}$. Then $\1_{(\C^N)^d} - t A(x)$ is a self-adjoint matrix for a.e.\ $x\in \R^d$ with eigenvalues contained in
\begin{equation*}
	[1 - t \Lambda(A), 1 - t \lambda(A)] = \big[- \tfrac{\Lambda(A) - \lambda(A)}{\Lambda(A)+ \lambda(A)}, \tfrac{\Lambda(A) - \lambda(A)}{\Lambda(A) + \lambda(A)} \big] = \big[ - \tfrac{1 - \varrho(A)}{1 + \varrho(A)}, \tfrac{1 - \varrho(A)}{1 + \varrho(A)} \big],
\end{equation*}
and thus
\begin{equation*}
	\d(A) \leq \tfrac{1 - \varrho(A)}{1 + \varrho(A)}
\end{equation*}
by the spectral radius formula.
\end{proof}


The next smoothing of the coefficients lemma will be important in Section~\ref{Section: OPR whole space} and Section~\ref{Section: OPR for bounded domains} to absorb terms, which are a priori not finite for non smooth coefficients. We include the simple proof for convenience.
To this end, we let $\eta \in \rC_{\cc}^{\infty}(B(0,1))$ be non-negative with $\int_{\R^d} \eta \, \d x = 1$ and put $\eta_n(x) \coloneqq n^d \eta(nx)$ for $n \in \N$ and $x \in \R^d$. We define the smoothed coefficients $A_n \coloneqq A * \eta_n$.

\begin{lemma} \label{Strongly elliptic systems: Lemma: Approximation of A}
Let $O \sub \R^d$ be open and bounded, $u \in \W^{1,2}(O)^N$ be a weak solution to $L u = 0$ in $O$ and $u_n \in \W^{1,2}(O)^N$ be the unique weak solution to
\begin{equation*}
    - \Div (A_n \nabla u_n) = 0 \quad \text{in} \; O \quad \& \quad u -u_n \in \W^{1,2}_0(O)^N.
\end{equation*}
Then the following assertions are satisfied.

\begin{enumerate}
    \item For all $n \in \N$ it holds $\lambda(A_n) \geq \lambda(A)$, $\Lambda(A_n) \leq \Lambda(A)$, $\d(A_n) \leq \d(A)$ and $u_n$ is smooth in $O$.
    \item Along a subsequence $u_n \to u$ in $\L^2(O)^N$ and a.e.\@ on $O$.
\end{enumerate}
\end{lemma}

\begin{proof}
As $A_n$ is smooth, so is $u_n$ by elliptic regularity theory, e.g.\ \cite[Sec.~6.3.1, Thm.~3]{Evans_PDE} adapted to systems. The rest of (i) follows from
\begin{equation*}
    A_n(x)\xi \cdot \overline{\zeta} = \int_{\R^d} \eta_n(y) A(x-y) \xi \cdot \overline{\zeta} \, \d y \qquad (\xi, \zeta \in (\C^{N})^d).
\end{equation*}
For instance, to prove that $\d(A_n) \leq \d(A)$, we let $t > 0$ be such that $\Vert \1 _{(\C^N)^d}-t A \Vert_{\infty} = \d(A)$ and use $\int_{\R^d} \eta_n(y) \, \d y =1$ twice in order to get for all $\xi \in (\C^N)^d$ that
\begin{equation*}
	|(\1_{(\C^N)^d} - t A_n(x)) \xi| = \left| \int_{\R^d} \eta_n(y) (\xi -t A (x-y) \xi) \, \d y \right|  \leq \d(A) \int_{\R^d} \eta_n(y) \, \d y = \d(A).
\end{equation*}
In order to prove (ii), let us show in a first step that $(v_n)_{n} \coloneqq (u_n - u)_{n} \sub \W^{1,2}_0(O)^N$ is bounded. Indeed, since
\begin{equation*}
    -\Div(A_n \nabla v_n) = \Div(A_n \nabla u) \quad \text{in} \; O,
\end{equation*}
this follows from the Lax-Milgram lemma and (i). Thus, we can find a subsequence $(v_k)_{k}$, and some $v \in \W^{1,2}_0(O)^N$ such that $v_k \to v$ weakly in $\W^{1,2}_0(O)^N$. By compactness, we can additionally assume $v_k \to v$ strongly in $\L^2(O)^N$ and a.e.\@ on $O$. Put $w \coloneqq v + u$. In particular, $u_k \to w $ in $\L^2(O)^N$ and a.e.\@ on $O$, and $\nabla u_k \to \nabla w$ weakly in $\L^2(O)^{dN}$. We claim that $w = u$. To this end, we fix some $\varphi \in \smooth[O]^N$. Then
\begin{equation*}
    0 = \int_O A_k \nabla u_k \cdot \overline{\nabla \varphi} \, \d x = \int_O \nabla u_k \cdot \overline{A_k^* \nabla \varphi} \, \d x \longrightarrow \int_O \nabla w \cdot \overline{A^* \nabla \varphi} \, \d x = \int_O A \nabla w \cdot \overline{\varphi} \, \d x,
\end{equation*}
using also strong $\L^2$-convergence $A_k^* \nabla \varphi \to A^* \nabla \varphi$, which follows from dominated convergence. This proves that $w$ solves
\begin{equation*}
    L w = 0 \quad \text{in} \; O \quad \& \quad w -u \in \W^{1,2}_0(O),
\end{equation*}
hence $w =u$.
\end{proof}

\begin{remark} \label{Strongly elliptic systems: Remark: A not defined on R^d}
If $A$ would be only defined on $O$, then we can extend it to $\R^d$ without changing the ``distance'': Simply let $t^* > 0$ be such that $\Vert \1_{(\C^N)^d} - t^* A \Vert_{\L^{\infty}(O)} = \d(A)$ and extend $A$ to $\R^d$ by $(t^*)^{-1} \1_{(\C^N)^d}$. Hence, we can always assume that $A$ is defined on $\R^d$.
\end{remark}

\section{New thoughts on old results of Koshelev} \label{Section: OPR whole space}

In a series of articles, culminating in the monograph~\cite{Koshelev}, Koshelev studied qualitative (H\"older) regularity of weak solutions to elliptic systems. In this section we explain how they lead us to an optimal perturbation result for Gaussian estimates for heat semigroups, when reinterpreted appropriately as quantitative statements.

\begin{definition}
Let $O \sub \R^d$ be open. We call a function $u \in \W^{1,2}(O)^N$ \textbf{$\boldsymbol{L}$-harmonic in $O$}, if we have for all $\varphi \in \smooth[O]^N$ that
\begin{equation*}
    \int_O A \nabla u \cdot \overline{\nabla \varphi} \, \d x =0.
\end{equation*}
\end{definition}

The appropriate setting to study regularity of $L$-harmonic functions turns out to be the following \textbf{weighted Morrey spaces} $\H_{\alpha}(O)^N$, $\alpha \in \R$, which are defined as the spaces of all $u \in \W^{1,2}(O)^N$ modulo $\C^N$ for which the norm
\begin{equation*}
    \Vert u \Vert_{\H_{\alpha}(O)} \coloneqq \sup_{x_0 \in O} \Vert u \Vert_{\H_{\alpha, x_0}(O)}, \quad \text{where} \quad \Vert u \Vert_{\H_{\alpha, x_0}(O)} \coloneqq \left( \int_{O} |\nabla u|^2 |x-x_0|^{\alpha} \, \d x \right)^{\frac{1}{2}},
\end{equation*}
is finite. For $\alpha > d-2$ sufficiently close to $d-2$ and $\varepsilon > 0$ small enough we have
\begin{equation}
 \label{eq: OPR whole space: Constant c(alpha,d)}
	c(\alpha, d, \varepsilon) \coloneqq \left( 1 + \frac{\alpha (d-2)}{d-1} + \varepsilon \right)^{\frac{1}{2}} \left( 1 - \frac{\alpha (\alpha - (d-2))}{2 (d-1)} - \varepsilon \right)^{-1} >0.
\end{equation}
This quantity will play an important role. In fact, $c(\alpha,d, \varepsilon) \to \delta(d)^{-1}$ in the limit as $\alpha \to d-2$ and $\varepsilon \to 0$. From now on we shall assume $\d(A) < \delta(d)$.

We begin by looking at $L$-harmonic functions on the unit ball $B$. Let $t > 0$. Guided by the perturbation principle in Lemma~\ref{Strongly elliptic systems: Lemma: q(A) vs. d(A)}, it begins with writing the equation $Lu = 0$ in $B$ in the weak sense as
\begin{equation} \label{eq: OPR whole space: Delta u = Div F}
    - \Delta u = - \Div(F) \quad \text{with} \quad F \coloneqq (\1_{(\C^{N})^d} - t A) \nabla u.
\end{equation}
Due to technical reasons we replace $A$ by $A_n$ and $u$ by $u_n$ as defined in Lemma \ref{Strongly elliptic systems: Lemma: Approximation of A} and call the term on the right-hand side $- \Div(F_n)$. In addition, we choose $t > 0$ such that $\d(A_n) = \Vert \1_{(\C^N)^d} - t A_n \Vert_{\infty}$.

Temporarily, fix $x_0 \in \frac{1}{4} B$. In order to derive optimal Morrey estimates for the solutions $u_n$, Koshelev considers two variational integrals
\begin{align*}
	X'(u_n,v_n) &\coloneqq \int_{\frac{1}{4}B} \nabla u_n \cdot \nabla v_n \, \d x,  \\
	Y_{-\alpha}'(v_n) &\coloneqq \int_{\frac{1}{4}B} |\nabla v_n|^2 |x - x_0|^{\alpha} \, \d x,
\end{align*}
where $\alpha>d-2$ is as above and $v_n$ is an ingeniously chosen test function for the equation on $\frac{1}{4}B$ that they constructs from $u_n$ using spherical harmonics \cite[Equ.~(2.3.2)]{Koshelev}. The precise formula for $v_n$ is not needed here -- it suffices to use the estimates below ``off-the-shelf''. In fact, this specific $v_n$ dates back to Giaquinta and Ne\v{c}as~\cite{Necas-Giaquinta}. Koshelev goes on by proving in \cite[Cor.~2.3.1]{Koshelev} the bounds
\begin{align*}
		X'(u_n,v_n)  &\geq \left( 1 - \frac{\alpha (\alpha - (d-2))}{2 (d-1)} - \varepsilon \right) \Vert u_n \Vert_{\H_{-\alpha, x_0} (\frac{1}{4} B)}^2  - C(\alpha, d, \varepsilon) \Vert \nabla u_n \Vert_{\L^2(\frac{1}{4} B)}^2,\\
		Y_{-\alpha}'(v_n) &\leq \left( 1 + \frac{(d-2) \alpha}{d-1} + \varepsilon \right) \Vert u_n \Vert_{\H_{-\alpha, x_0}(\frac{1}{4}B)}^2 + C(\alpha,d, \varepsilon) \Vert \nabla u_n \Vert_{\L^2(\frac{1}{4}B)}^2.
\end{align*}
Since $v_n$ is a test function for the equation for $u_n$ in $\frac{1}{4}B$, we have
\begin{align*}
	X'(u_n,v_n)
	=  \int_{\frac{1}{4}B} (\1_{(\C^{N})^d} - t A_n) \nabla u_n \cdot \nabla v_n \, \d x
\end{align*}
and the Cauchy--Schwarz inequality along with the bound $\d(A_n) \leq \d(A)$ in Lemma~\ref{Strongly elliptic systems: Lemma: Approximation of A} (i) yields
\begin{align*}
	|X'(u_n,v_n)| \leq \d(A) \|u_n \|_{\H_{-\alpha, x_0}(\frac{1}{4}B)} Y'_{-\alpha}(v_n)^{\frac{1}{2}}.
\end{align*}
Combining the previous three estimates and recalling the definition of  $c(\alpha, d, \varepsilon)$ leads to
\begin{align}
	\label{eq: OPR whole space: not absorbed, with F}
\begin{split}
	\Vert u_n \Vert_{\H_{-\alpha, x_0}(\frac{1}{4}B)}^2
	&\leq \d(A) \Vert u_n \Vert_{\H_{-\alpha, x_0}(\frac{1}{4}B)} \Big[ (c(\alpha,d, \varepsilon) \Vert u_n \Vert_{\H_{-\alpha, x_0}(\frac{1}{4}B)}
	\\&\quad+ C(\alpha,d, \varepsilon) \Vert \nabla u_n \Vert_{\L^2(\frac{1}{4}B)} \Big] + C(\alpha,d,\varepsilon) \|\nabla u_n \|_{\L^2(\frac{1}{4}B)}^2,
\end{split}
\end{align}
where $C(\alpha, d, \varepsilon)$ varies from line to line. The smoothing of the coefficients guarantees that the first summand on the right-hand side is finite and this is the very reason why we have to include this argument. By Young's inequality, it follows that
\begin{equation}
   \Vert u_n \Vert_{\H_{-\alpha, x_0}(\frac{1}{4}B)}^2
   \leq (c(\alpha,d, \varepsilon) + \varepsilon) \, \d(A) \Vert u_n \Vert_{\H_{-\alpha, x_0}(\frac{1}{4}B)}^2 + C(\alpha,d,\varepsilon) \|\nabla u_n \|_{\L^2(\frac{1}{4}B)}^2.
\end{equation}
Since $\d(A) < \delta(d)$, we can fix $\varepsilon$ small and $\alpha$ close to $d-2$ depending only on dimension and $\d(A)$ such that the first term on the right can be absorbed. This is the key point and the result is
\begin{equation}
    \Vert u_n \Vert_{\H_{-\alpha, x_0}(\frac{1}{4}B)}
   \les_{d, \d(A)} \| \nabla u_n \|_{\L^2(\frac{1}{4}B)}.
\end{equation}

After taking the supremum over all $x_0 \in \frac{1}{4} B$ we arrive at
\begin{equation} \label{eq: OPR whole space: already absorbed}
	\Vert u_n \Vert_{\H_{-\alpha}(\frac{1}{4}B)}
	\les_{d, \d(A)} \| \nabla u_n \|_{\L^2(\frac{1}{4}B)}.
\end{equation}
Then, Koshelev proves in  \cite[Thm.\@ 2.1.1]{Koshelev} that the left-hand side controls the H\"older seminorm of order $\mu = \nicefrac{(\alpha - d+2)}{2}$ on $\frac{1}{4}B$. Applying Caccioppoli's inequality on the right-hand side eventually leads to
\begin{equation*}
    [u_n]^{(\mu)}_{\frac{1}{4}B} \coloneqq \sup_{x,y \in \frac{1}{4} B, x \neq y} \frac{|u_n(x) - u_n(y)|}{|x-y|^{\mu}} \les_{d, \d(A)} \Vert u_n \Vert_{\L^2(B)}.
\end{equation*}
Finally, we invoke Lemma~\ref{Strongly elliptic systems: Lemma: Approximation of A} (ii) in order to deduce
\begin{equation*}
   [u]^{(\mu)}_{\frac{1}{4}B} \leq \limsup_{n \to \infty} \, [u_n]^{(\mu)}_{\frac{1}{4}B} \les_{d, \d(A)} \Vert u \Vert_{\L^2(B)}.
\end{equation*}
This estimate holds for any $L$-harmonic function on the unit ball. Since we have
\begin{equation*}
    \d(A) = \d(A^*) \quad \& \quad \d(A) = \d(A(x_0 + r \, \cdot))
\end{equation*}
for each $x_0 \in \R^d$ and $r > 0$, a scaling argument shows that the outcome of revisiting Koshelev's results is the following proposition.

\begin{proposition} \label{OPR whole space: Proposition: H(mu)}
Suppose $\d(A) < \delta(d)$. There are $\mu \in (0,1]$ and $C > 0$, both depending only on $d$ and $\d(A)$, such that we have for all balls $B = B(x,r) \sub \R^d$ and every $L$- or $L^*$-harmonic $u$ in $B$ that
\begin{equation*}
     r^{\mu} [u]_{\frac{1}{4} B}^{(\mu)} \leq C r^{- \frac{d}{2}} \Vert u \Vert_{\L^2(B)}.
\end{equation*}
\end{proposition}

The quantitative H\"older estimate in Proposition~\ref{OPR whole space: Proposition: H(mu)} appeared much later in a different context. Namely, Auscher and Tchamitchian \cite{Auscher_Heat-Kernel, Auscher_Tchamitchian_Kato} called it property $(\H)$ and proved that it implies that $\e^{-tL}$ has a kernel with pointwise Gaussian bounds. If we combine their Theorem 10 in \cite[Chap.\@ 1, Sec.\@ 1.4.1]{Auscher_Tchamitchian_Kato} with Proposition~\ref{OPR whole space: Proposition: H(mu)} above, then we obtain the following result for Gaussian estimates. It can be seen as a perturbation result from the Laplacian.

\begin{theorem}
\label{OPR whole space: Theorem: H(mu) implies G(mu)}
Suppose $\d(A) < \delta(d)$. The kernel of $(\e^{-t L})_{t >0}$ is represented by a H\"older regular function $(K_t)_{t > 0}$, which admits pointwise Gaussian estimates: There are $c, a > 0$ and $\mu \in (0,1)$ such that
\begin{align*}
    |K_t(x,y)| & \leq c t^{- \frac{d}{2}} \e^{-a \frac{|x-y|^2}{t}}, \\
    |K_t(x,y)- K_t(x',y')| &\leq c t^{-\frac{d}{2} - \frac{\mu}{2}} ( |x-x'| + |y-y'| )^{\mu}
\end{align*}
for all $t>0$ and $x,x',y,y' \in \R^d$. The constants $c, a, \mu$ depend only on $d$, $\lambda(A)$ and $\Lambda(A)$.
\end{theorem}


From Young's inequality for convolutions, we obtain:

\begin{corollary}
\label{OPR whole space: Corollary: H(mu) implies G(mu)}
If $\d(A) < \delta(d)$, then $p_+(L) = \infty$.
\end{corollary}

We shall see next that the ``radius'' $r=\delta(d)$ is optimal for the conclusion in Corollary~\ref{OPR whole space: Corollary: H(mu) implies G(mu)} and hence also for the one in Theorem~\ref{OPR whole space: Theorem: H(mu) implies G(mu)}. Again this is implicit in Koshelev's work and relies on a counterexample due to De Giorgi.

Let $c > 0$ and $D \geq \nicefrac{(c^2 +1)}{(d-2)c}$. Define for $x \in \R^d \setminus \{ 0 \}$ the elliptic system with coefficients
\begin{equation*}
    (A_{\mathrm{DG}}(x))_{i,j}^{\alpha, \beta} \coloneqq \left( c \delta_{ij} + D \frac{x_i x_j}{|x|^2} \right) \left( c \delta_{\alpha \beta} + D \frac{x_\alpha x_\beta}{|x|^2} \right) \qquad (i,j, \alpha, \beta = 1, \up, d).
\end{equation*}
Then $u(x) \coloneqq \nicefrac{x}{|x|^{b}}$ with
\begin{equation*}
    b = \frac{d}{2} - \left( \frac{d^2}{4} - \frac{d(d-1) c D + (d-1) D^2}{1 + (c + D)^2}  \right)^{\frac{1}{2}} \in [1, \nicefrac{d}{2})
\end{equation*}
solves the elliptic system $- \Div (A_{\mathrm{DG}} \nabla u) =0$ in the weak sense in $B(0,1)$, see  \cite[Sec.\@ 2.5]{Koshelev} or \cite{DeGiorgi_counterexample}. Note that $b =1$ if and only if $D = \nicefrac{(c^2 +1)}{(d-2) c}$. Koshelev continues in \cite[Sec.\@ 2.5]{Koshelev} by showing that for this choice of $D$ he can pick $c = c(d) > 0$ such that $\d(A_{\mathrm{DG}}) = \delta(d)$. Since $\d(A_{\mathrm{DG}})$ depends continuously on $D$ and $c$, we can pick these parameters for any given $\varepsilon > 0$ in such a way that
\begin{align}
\label{OPR whole space: Parameters for counterexample}
b > 1 \quad \& \quad \d(A_{\mathrm{DG}}) < \delta(d) +\varepsilon.
\end{align}
Now we use a localization argument from \cite[Chap.\@ 1, Sec.\@ 1.3]{Auscher_Tchamitchian_Kato} to prove:

\begin{proposition}
\label{OPR whole space: Proposition: p+ optimal}
For any $\varepsilon > 0$ there is $A$ such that $\d(A) < \delta(d) +\varepsilon$ and $p_+(L) < \infty$.
\end{proposition}

\begin{proof}
We pick $A = A_{\mathrm{DG}}$ as in \eqref{OPR whole space: Parameters for counterexample} and set $u(x) \coloneqq \nicefrac{x}{|x|^{b}}$. Let $\phi \in \rC^{\infty}_{\cc}$ be such that $\1_{B(0, \nicefrac{1}{2})} \leq \phi \leq \1_{B(0,1)}$. Then $v \coloneqq \phi u \in \D(L)$ and using $Lu =0$, we deduce
\begin{equation*}
	Lv = -\Div ( A (\nabla \phi \otimes u)) - \Div (\phi A \nabla u)= -\Div ( A (\nabla \phi \otimes u)) - \nabla \phi \odot (A \nabla u).
\end{equation*}
Our notation should be interpreted as what comes out from the product rule. The only thing that matters is that $\nabla \phi$ vanishes near the origin and hence $Lv \in  (\rC^{\infty}_{\cc})^N$. If we had $p_+(L) = \infty$, then according to \cite[Prop.\@ 5.3]{Auscher_Riesz_Trafo} we would get $v \in  (\L^q)^N$ for every $q \in (2^*,\infty)$. However, $|v(x)|= |x|^{1- b}$ in a neighborhood of $0$ can not belong to $\L^q$ for $q \geq \nicefrac{d}{(b -1)}$.
\end{proof}

\section{The interpolation argument} \label{Section: A priori estimates: Interpolation approach}

We come to the proof of our main result, Theorem~\ref{Introduction: Theorem: Application of the abstract interpolation result}, for the case $\d(A) \geq \delta(d)$. We will use basic properties of semigroups and vector-valued holomorphic functions. For further background we refer to \cite{VVLT}.

We begin with a Stein-type interpolation principle tailored to our needs (taking care of implicit constants in particular). We write
\begin{align*}
	\S_{\delta} &\coloneqq \{ z \in \C : - \delta < \re(z) < 1 + \delta \} \quad \& \quad \S \coloneqq \S_0.
\end{align*}
For a matrix $A=A_z$ depending on a parameter $z$ we let $a_z$ be the sesquilinear form corresponding to $L_z \coloneqq - \Div (A_z \nabla \cdot )$.

\begin{proposition} \label{Introduction: Proposition: A priori bound via interpolation}
Let $\delta > 0$. Suppose that $\{ A_{z} \}_{z \in \S_{\delta}} \sub \L^{\infty}(\R^d; \LL((\C^N)^d))$ are uniformly strongly elliptic matrices such that:
\begin{enumerate}
\item There are $0 < \lambda \leq \Lambda$ with $\lambda(A_z) \geq \lambda$ and $\Lambda(A_z) \leq \Lambda$ for all $z \in \overline{\S}$.

\item We have $\sup_{t \in \R} \d(A_{\mathrm{i} t}) < \delta(d)$.

\item For all $u,v \in (\W^{1,2})^N$ the map $z \mapsto a_z(u,v)$ is holomorphic in $\S_{\delta}$.
\end{enumerate}
Then for $\theta \in [0,1]$ we have $p_+(L_{\theta}) \geq \nicefrac{2^*}{\theta}$.
\end{proposition}

\begin{proof}
Fix $t > 0$ and $\theta \in [0,1]$. Let us define $\Phi(z) \coloneqq \e^{- t L_{z}}$ for $z \in \S_{\delta}$. As an  $\LL((\L^2)^N)$-valued map, $\Phi$ is bounded by $1$, holomorphic on $\S_\delta$ and in particular continuous on $\overline{\S}$. This follows from (iii), see \cite{Hol_dependence}. Now, we estimate $\Phi$ on the boundary of $\S$. Let $f \in (\L^2)^N$.

\textbf{(A) Estimate on the left boundary.} Let $z \in \overline{\S}$ with $\re (z) = 0$. Due to (ii) and Theorem \ref{OPR whole space: Theorem: H(mu) implies G(mu)} we have Gaussian estimates for the kernel $(K_{t,z})_{t > 0}$ of $(\e^{-t L_z})_{t >0}$ at our disposal. By (i) implicit constants depend only on $\lambda, \Lambda$ and $d$. Young's inequality for convolutions yields
\begin{equation*}
    \Vert \e^{- t L_z} f \Vert_{\infty} \les_{d, \lambda, \Lambda} t^{-\frac{d}{4}} \Vert f \Vert_2.
\end{equation*}
\textbf{(B) Estimate on the right boundary.} Let $z \in \overline{\S}$ with $\re (z) = 1$. By holomorphy of the semigroup it follows that $\e^{- t L_z } f \in (\W^{1,2})^N$. Thus, by a Sobolev embedding, (i) and H\"older's inequality, we get
\begin{align*}
\Vert \e^{- t L_z} f \Vert_{2^*}^2
&\les_d \Vert \nabla \e^{- t L_z} f \Vert_{2}^2 \\
&\les_{\lambda} |a_z (\e^{- t L_z} f, \e^{- t L_z} f)| \\
&= |(L_z \e^{- t L_z} f \, | \, \e^{- t L_z} f)_2| \\
&\leq \Vert L_z \e^{- t L_z} f \Vert_{2} \Vert \e^{- t L_z} f \Vert_{2}.
\end{align*}
The semigroup is contractive on $(\L^2)^N$ in the sector $\{ z \in \C: |\arg z| < \varphi \}$, where $\tan \varphi = \nicefrac{\lambda}{\Lambda}$. In particular, $\Vert \e^{- t L_z} f \Vert_{2} \leq \Vert f \Vert_{2}$ and by Cauchy's formula for the complex derivative
\begin{equation*}
    \Vert L_z \e^{-t L_z} f \Vert_{2} = \Big \Vert \frac{1}{2 \pi \mathrm{i}} \int_{\partial B(t, r)} \frac{\e^{- w L_z} f}{(w -t)^2} \, \d w \Big \Vert_{2} \les_{\lambda, \Lambda} t^{-1} \Vert f \Vert_2
\end{equation*}
with $r = t \sin (\varphi)$.
Altogether, this gives
\begin{equation*}
    \Vert \e^{-t L_z} f \Vert_{2^*} \les_{\lambda, \Lambda,d} t^{-\frac{1}{2}} \Vert f \Vert_2.
\end{equation*}

Combining (A) and (B), Stein's Interpolation Theorem \cite[Thm.\@ 1]{Stein_Interpolation} implies that
\begin{equation*}
\Vert \e^{-t L_{\theta}} f \Vert_{\frac{2^*}{\theta}} \les_{d, \lambda, \Lambda, \theta} t^{\frac{d}{\nicefrac{2 \cdot 2^*}{ \theta}} - \frac{d}{4}} \Vert f \Vert_{2}.
\end{equation*}
In the language of $\L^p - \L^q$-estimates, this means that $(\e^{-t L_{\theta}})_{t > 0}$ is $\L^2 - \L^{\nicefrac{2^*}{\theta}}$-bounded. A general principle for these estimates (\cite[Prop.\@ 3.2]{Auscher_Riesz_Trafo}) implies $\Vert \e^{-t L_{\theta}} f \Vert_q \les_{d, \lambda, \Lambda, q} \Vert f\Vert_q$ for all $q \in (2,\nicefrac{2^*}{\theta})$ as claimed.
\end{proof}

\begin{proof}[\rm\bf{Proof of Theorem \ref{Introduction: Theorem: Application of the abstract interpolation result} (i).}]
Let $A$ be elliptic such that $\d(A) \in [\delta(d),1)$. Fix $\varepsilon \in (0, 1-\d(A))$ and $t^* > 0$ such that $\Vert \1_{(\C^N)^d} - t^* A \Vert_{\infty} = \d(A)$. We abbreviate $B\coloneqq \1_{(\C^N)^d} - t^* A$, which means $A = (t^*)^{-1} (\1_{(\C^N)^d} - B)$. As sketched in the introduction, we perturb $B$ by multiplication with complex numbers from a suitable annulus. However, it will be convenient to parametrize these numbers via an analytic function $F$ defined on $\S$.

Let $0 < r < 1 < R$ to be chosen. We embed $A$ into the analytic family
	\begin{equation*}
		A_{z} \coloneqq (t^*)^{-1}(\1_{(\C^N)^d} - F(z) B), \quad \text{where} \quad F(z) \coloneqq r^{1-z}R^z = r \e^{z \ln(\nicefrac{R}{r})}.
	\end{equation*}
	Note that $F$ is entire, bounded by $r^{-\delta}R^{1+\delta} = \delta(d)^{-\delta} R$ in any strip $\S_\delta$, and maps $\overline{\S}$ onto $\{ z \in \C : r \leq |z| \leq R \}$. At this point we choose
	\begin{equation*}
		r \coloneqq \frac{\delta(d)}{\d(A)+ \varepsilon} \quad \& \quad R \coloneqq \frac{1}{\d(A) + \varepsilon}.
	\end{equation*}
    Let us show that for $z$ in a strip $\S_\delta$ with sufficiently small $\delta>0$ we can define the interpolating operators $L_{z} \coloneqq - \Div (A_{z} \nabla \, \cdot)$, where $A_{z}$ is still elliptic by Lemma~\ref{Strongly elliptic systems: Lemma: q(A) vs. d(A)}. Indeed, for $\delta > 0$ sufficiently small our choice of $R$ delivers
\begin{equation*}
	\sup_{z \in \S_{\delta}} \Vert A_z \Vert_{\infty} \leq (t^*)^{-1}(1 + \delta(d)^{-\delta} R \, \d(A)) \quad \& \quad \sup_{z \in \S_{\delta}} \d(A_z) \leq \delta(d)^{-\delta} R \, \d(A) < 1.
\end{equation*}
This also proves (i) in Proposition~\ref{Introduction: Proposition: A priori bound via interpolation}. Part (iii) follows immediately and our choice of $r $ yields
\begin{equation*}
\sup_{t \in \R} \d( A_{\mathrm{i} t} ) \leq \sup_{t \in \R} \Vert F(\mathrm{i} t) B \Vert_{\infty} = r \, \d (A) < \delta(d),
\end{equation*}
which is (ii). Now, pick $\theta \in (0,1)$ such that $1 = r^{1- \theta} R^{\theta}$. Then $L = L_{\theta}$ and Proposition~\ref{Introduction: Proposition: A priori bound via interpolation} implies that $p_+(L) \geq \nicefrac{2^*}{\theta}$. Finally, we notice that
\begin{equation*}
\theta =1 - \frac{ \ln(R)}{\ln(\nicefrac{R}{r})} \longrightarrow 1 - \frac{\ln (\d(A))}{\ln(\delta(d))} \quad \text{as} \; \varepsilon \searrow 0.  \qedhere
\end{equation*}
\end{proof}

\begin{remark} \label{A priori bound for p_+(L) via interpolation: Remark: Advantages}
In Proposition~\ref{Introduction: Proposition: A priori bound via interpolation} we assume that $z \mapsto a_z(u,v)$ is holomorphic in a larger strip for convenience to get continuity of $z \mapsto \e^{-t L_z}$ up to $\overline{\S}$. If this holds true for any other reason, it is enough to suppose that $z \mapsto a_z (u,v)$ is holomorphic in $\S$.
\end{remark}

\begin{remark} \label{A priori bound for p_+(L) via interpolation: Remark: Advantages}
The proof of Theorem~\ref{Introduction: Theorem: Application of the abstract interpolation result} reveals that the same results hold for divergence form operators with form domain $V$ on general open sets $O \sub \R^d$, provided that Theorem~\ref{OPR whole space: Theorem: H(mu) implies G(mu)} holds true with implicit constants depending only on geometry, ellipticity and dimension, and that we have additionally
    \begin{equation*}
        \Vert u \Vert_{\L^{2^*}(O)} \les \Vert \nabla u \Vert_{\L^2(O)} \qquad (u \in V),
    \end{equation*}
which was used in (B) above.
\end{remark}

\section{Extension to bounded $\rC^1$-domains}
\label{Section: OPR for bounded domains}

Let us extend Theorem~\ref{Introduction: Theorem: Application of the abstract interpolation result} and Theorem~\ref{OPR whole space: Theorem: H(mu) implies G(mu)} to bounded $\rC^1$-domains with Dirichlet boundary conditions. The divergence form operator $L = -\Div(A \nabla \, \cdot)$ with uniformly strongly elliptic $A \in \L^{\infty}(\Omega; \LL((\C^N)^d))$ is now realized in $\L^2(\Omega)^N$ as the m-accretive operator associated to the form
\begin{equation*}
    a(u,v) \coloneqq \int_{\Omega} A \nabla u \cdot \overline{\nabla v} \, \d x \qquad (u,v \in \W^{1,2}_0(\Omega)^N).
\end{equation*}
We fix our geometric setup.

\begin{assumption} \label{Assumption: Geometry for OPR}
Throughout this section $\Omega \sub \R^d$, $d \geq 3$, is a bounded domain with $\rC^1$-boundary. This means that there is some $M > 0$ such that for each $x_0 \in \partial \Omega$ there is an open neighborhood $U$ of $x_0$ and a $\rC^1$-diffeomorphism $\phi \colon U \to B(0,1)$,
$\phi(x) = (x', \psi(x') - x_d)$
such that $\phi(U \cap \Omega) = B(0,1) \cap \R^d_+$ and $\Vert D \psi \Vert_{\infty} \leq M$.
\end{assumption}

We can choose $M$ arbitrarily small by choosing the neighborhoods small enough. This is exactly the reason, why we assume that the boundary is $\rC^1$ and not just Lipschitz.

\begin{theorem}
In the setting above suppose that $\d(A) < \delta(d)$. Then the kernel of $(\e^{-t L})_{t >0}$ is represented by a measurable function $(K_t)_{t > 0}$ for which there are $c, a > 0$ and $\mu \in (0,1)$ such that
\begin{align*}
	|K_t(x,y)| & \leq c t^{- \frac{d}{2}} \e^{-a \frac{|x-y|^2}{t}}, \\
	|K_t(x,y)- K_t(x',y')| &\leq c t^{-\frac{d}{2} - \frac{\mu}{2}} ( |x-x'| + |y-y'| )^{\mu}
\end{align*}
for all $t>0$ and $x,x',y,y' \in \Omega$. The constants $c$ and $a$ depend only on $d, \lambda(A), \Lambda(A)$ and geometry.
\end{theorem}

\begin{proof}
The proof is very similar to the one of Theorem \ref{OPR whole space: Theorem: H(mu) implies G(mu)}. We abbreviate
\begin{equation*}
    \Omega(x,r) \coloneqq \Omega \cap B(x,r).
\end{equation*}
We use again elliptic estimates for $L$-harmonic functions due to Koshelev, this time also in the half space after localization and transformation. The kernel estimates will then follow from \cite[Thm.\@ 12]{Auscher_Tchamitchian_Domains} provided we can check what they call property (D)\footnote{Note carefully that the bounded $\rC^1$-domain $\Omega$ falls into class $(\mathrm{II})$ in \cite{Auscher_Tchamitchian_Domains}. In this case, and since we consider Dirichlet boundary conditions, the properties $(\mathrm{D})$ and $(\mathrm{D_{loc}})$ in \cite{Auscher_Tchamitchian_Domains} coincide by definition.} for $L$ and $L^*$. By the easy argument in \cite[p.~37]{Auscher_Tchamitchian_Kato} it suffices to show a property similar to Proposition~\ref{OPR whole space: Proposition: H(mu)} and formulated as follows:

There are $C>0$ and $\gamma \in (0,1)$ such that for all $x_0 \in \overline{\Omega}$, $r \leq \rho_0$ and $u \in \W^{1,2}_0(\Omega)^N$ with $L u = 0$ or $L^* u =0$ in $\Omega(x_0,r)$ it holds
\begin{equation}
\label{eq: Goal C1 case}
    r^{\mu} [u]^{(\mu)}_{\Omega(x_0, \gamma r)} \leq C r^{-\frac{d}{2}} \Vert u \Vert_{\L^2(\Omega(x_0,r))}.
\end{equation}
Here $\rho_0 \leq 1$ is chosen small as explained in \cite[p.\@ 20]{Auscher_Tchamitchian_Domains}.

In view of $\d(A) = \d(A^*)$, we stick to the case of $L$-harmonic functions. When $x_0 \in \Omega$ and $B(x_0,r) \subseteq \Omega$, this estimate has already been obtained in Section~\ref{Section: OPR whole space} with $\gamma = \nicefrac{1}{4}$. By a case distinction (whether or not $B(x_0, \nicefrac{r}{2})$ intersects $\partial \Omega$) it suffices to treat in addition the case $x_0 \in \partial \Omega$.

So, let $r \leq \rho_0$ and $u \in \W^{1,2}_0(\Omega)^N$ with $L u =0$ in $\Omega(x_0,r)$.
Let $t > 0$ and pick $\rho \simeq_M r$ small enough such that $\phi^{-1}(B(0, \rho)) \sub B(x_0,r)$. Put $B_+ \coloneqq \R^d_+ \cap B$, where $B$ is again the unit ball, and write $u_{\phi, \rho} \coloneqq (u \circ \phi^{-1})(\rho \, \cdot )$. A change of coordinates implies that $u_{\phi, \rho} \in \W^{1,2}(B_+)^N$ is a weak solution of
\begin{equation*}
- \Div\Big((D \phi)_{\phi, \rho} A_{\phi, \rho} (D \phi)_{\phi, \rho}^{\top} \nabla (u_{\phi, \rho})\Big) = 0 \quad \text{in} \; B_+
\end{equation*}
that vanishes on $\partial \R^d_+ \cap \partial B_+$.

Next, we use a smoothing procedure as in Section~\ref{Section: OPR whole space}. We let $A_n$ and $u_n$ be defined as in Lemma~\ref{Strongly elliptic systems: Lemma: Approximation of A} with $A$ replaced by $(D \phi)_{\phi, \rho} A_{\phi, \rho} (D \phi)_{\phi, \rho}^{\top}$ (see also Remark~\ref{Strongly elliptic systems: Remark: A not defined on R^d}). Then the same lemma assures that $u_n \to u_{\phi, \rho}$ in $\L^2(B_+)^N$ and a.e.\@ along a subsequence. In addition, we have $u_n \in \rC^{\infty}(\overline{\frac{3}{4} B_+})^N$ by elliptic regularity \cite[Sec.\@ 6.3.1, Thm.\@ 5]{Evans_PDE}. As in Section~\ref{Section: OPR whole space}, we write $- \Div (A_n \nabla u_n) =0$ as
\begin{equation*}
    - \Delta u_n = - \Div (F_n) \quad \text{with} \quad F_n \coloneqq (\1_{(\C^N)^d} - t A_n) \nabla u_n
\end{equation*}
and $t > 0$ is chosen such that $\Vert \1_{(\C^N)^d} - t A_n \Vert_{\infty} = \d(A_n)$. Note that
\begin{equation}
\label{Eq: Sebastian}
    \Vert \1_{(\C^N)^d} - t A_n \Vert_{\infty}
    \leq \d((D \phi)_{\phi, \rho} A_{\phi, \rho} (D \phi)_{\phi, \rho}^{\top})
    \leq M^2 + (1+M)^2 \d(A),
\end{equation}
where the first inequality is due to Lemma~\ref{Strongly elliptic systems: Lemma: Approximation of A} (i) and the second one follows by definition of $\phi$.
Koshelev proves in \cite[(2.4.13)]{Koshelev} for each $x_0 \in \frac{1}{4} B_+$ that
\begin{align} \label{eq: OPR domain: not absorbed, f not replaced}
     \Vert u_n \Vert_{\H_{-\alpha,x_0}(\frac{1}{4} B_+)} &\leq c(\alpha,d, \varepsilon) \Vert F_n |x-x_0|^{-\frac{\alpha}{2}} \Vert_{\L^2(\frac{1}{4} B_+)}
  \\& \quad + C(\alpha, \varepsilon,d) \Big[\Vert \nabla u_n \Vert_{\L^2(\frac{1}{4} B_+)} + \Vert F_n \Vert_{\L^2(\frac{1}{4} B_+)}\Big], \notag
\end{align}
provided that the right-hand side is finite, and $u_n$ and $f_n$ are sufficiently smooth, which is the case thanks to our smoothing procedure. Here, $c(\alpha,d, \varepsilon)$ is as in \eqref{eq: OPR whole space: Constant c(alpha,d)}, where $\alpha > d -2$ and $\varepsilon > 0$ are chosen such that $c(\alpha,d, \varepsilon)$ is positive and finite. By definition of $F_n$ and \eqref{Eq: Sebastian}, we derive the estimate
\begin{align}
  \Vert u_n \Vert_{\H_{-\alpha,x_0}(\frac{1}{4} B_+)} &\leq \Big((1 + M)^2 \d(A) + M^2\Big) \Big( c(\alpha, d, \varepsilon) + \varepsilon \Big) \Vert u_n \Vert_{\H_{-\alpha, x_0}(\frac{1}{4} B_+)}
  \\&\quad + C(\alpha, \varepsilon,d, M) \Vert \nabla u_n \Vert_{\L^2(\frac{1}{4} B_+)}.
\end{align}
As $\d(A) < \delta(d)$, we can pick $\varepsilon > 0$, $\alpha > d-2$ and $M > 0$ depending only on $\d(A)$ and $d$ such that first term on the right can be absorbed in order to obtain
\begin{equation} \label{eq: OPR domain: already absorbed}
    \Vert u_n \Vert_{\H_{-\alpha}(\frac{1}{4} B_+)} \les_{d, \d(A)} \Vert \nabla u_n \Vert_{\L^2(\frac{1}{4} B_+)}.
\end{equation}
As in Section~\ref{Section: OPR whole space} we
deduce
\begin{equation*}
    [ u_{\phi, \rho} ]_{\frac{1}{4} B_+}^{(\mu)} \les_{d, \d(A)} \Vert u_{\phi, \rho} \Vert_{\L^2(B_+)},
\end{equation*}
where $\mu \coloneqq \nicefrac{(\alpha - d +2)}{2}$.

Transforming back gives \eqref{eq: Goal C1 case} for some $\gamma = \gamma(M) \in (0,1)$.
\end{proof}

At this point we are in the same situation as on $\R^d$ and we can derive the following statement.

\begin{corollary}
In the setup of this section the following assertions hold true.
\begin{enumerate}
    \item If $\d(A) > \delta(d)$, then
    \begin{equation*}
    p_+(L) \geq \frac{2^*}{1 - \frac{\ln(\d(A))}{\ln(\delta(d))}}.
    \end{equation*}

    \item If $\d(A) \leq \delta(d)$, then $p_+(L) = \infty$.

    \item Part (ii) is sharp in the sense that for all bounded $\rC^1$-domains $\Omega \sub \R^d$ and for each $\varepsilon >0$ there is some $A_{\varepsilon}$ with $\d(A_{\varepsilon}) \leq \delta(d) + \varepsilon$ and $p_+(L_{\varepsilon}) < \infty$.
\end{enumerate}
\end{corollary}

\begin{proof}
The estimates for $p_+(L)$ follow as before, see also Remark~\ref{A priori bound for p_+(L) via interpolation: Remark: Advantages}. As for the sharpness of the radius $\d(A) = \delta(d)$ we can, after translation, assume $0 \in \Omega$. We take the same coefficients $A=A_\mathrm{DG}$ as in the proof of Proposition~\ref{OPR whole space: Proposition: p+ optimal} and localize $u$ to a ball contained in $\Omega$. As before, this produces some $v \in \D(L)$ with $v \notin (\L^q)^N$ for $q$ large but $L v \in (\rC_{\cc}^{\infty})^N$. Arriving at a contradiction with $p_+(L) = \infty$ requires a different (and in fact simpler) argument compared to the case $\Omega = \R^d$.

By ellipticity and Poincaré's inequality, there is some $\theta_2 >0$ such that $L-\theta_2$ is still m-accretive. Hence, $L$ is invertible in $(\L^2)^N$ and the semigroup enjoys the exponential bound $\|\e^{-tL}f\|_2 \leq \e^{-\theta_2 t}\|f\|_2$ for all $t>0$ and $f \in (\L^2)^N$. By interpolation with the uniform bound on $(\L^p)^N$ for some $p>q$, we get $\|\e^{-tL}f\|_q \les \e^{-\theta_q t}\|f\|_q$ with some $\theta_q > 0$. But then the formula
\begin{equation*}
    L^{-1}f = \int_0^\infty \e^{-tL}  f \, \d t,
\end{equation*}
valid in $(\L^2 \cap \L^q)^N$ by the exponential estimate, implies that $L^{-1}$ maps $(\L^q \cap \L^2)^N$ into itself, in contradiction with the properties of $v$.
\end{proof}

\section{Dimensionless improvements}
\label{Section: dimensionless}

Here, we prove Theorem~\ref{Introduction: Theorem: q+ improvement}. For $1<p<\infty$ we denote by $(\Wdot^{1,p})^N$ the space of all $\C^N$-valued tempered distributions modulo $\C^N$ for which the distributional gradient belongs to $(\L^p)^{dN}$. We endow this space with the norm $\Vert \nabla \cdot \Vert_p$ and denote by $(\Wdot^{-1,p})^N$ the anti-dual space of $(\Wdot^{1, p'})^N$. We define
\begin{equation}
\label{eq: Dimensionless improvements: divAgrad}
    -\Div A \nabla \colon (\Wdot^{1,p})^N \to (\Wdot^{-1,p})^N, \quad \langle -\Div A \nabla u \, | \, v \rangle \coloneqq \int_{\R^d} A \nabla u \cdot \overline{\nabla v} \, \d x.
\end{equation}
By H\"older's inequality, this is a bounded map.

We denote by $c(p)$ the operator norm of the Riesz transform $R \coloneqq \nabla ( - \Delta )^{- \nicefrac{1}{2}}$ (defined via a Fourier multiplication operator with symbol $- \mathrm{i} \frac{\xi}{|\xi|} \otimes \1_{\C^N}$) from $(\L^p)^N$ to $(\L^p)^{dN}$. By Plancherel's theorem we have $c(2)=1$.

\begin{proposition}
\label{Dimensionless improvements for the critical numbers: Proposition: Memoirs extrapolation}
Let $p > 2$. Then $c(p) < \nicefrac{1}{\sqrt{\d(A)}}$ implies $q_+(L) \geq p$.
\end{proposition}

\begin{proof}
By the characterization of $q_+(L)$ in \cite[Sec.\@ 13.3]{Auscher-Egert}, it suffices to prove that the map in \eqref{eq: Dimensionless improvements: divAgrad} is invertible and that the inverse is compatible with the one for $p=2$.

We borrow an idea from \cite[Lem.\@ 3.4]{Auscher_Riesz_Trafo}. Fix $t^* > 0$ such that $\d(A) = \Vert \1_{(\C^N)^d} - t^* A \Vert_{\infty}$ and put $B \coloneqq \1_{(\C^N)^d} - t^* A$. Then we can factorize
\begin{align*}
    - \Div A \nabla &= (t^*)^{-1} (-\Delta) + (t^*)^{-1} \Div B \nabla
    \\&= (t^*)^{-1} (- \Delta)^{\frac{1}{2}} \left(\1 + R^* B R \right) (- \Delta)^{\frac{1}{2}}.
\end{align*}
Since
\begin{equation*}
    \Vert R^* B R f \Vert_p \leq c(p)^2 \d(A) \Vert f \Vert_p
\end{equation*}
for $f \in (\L^p)^N \cap (\L^2)^N$, it suffices that $c(p) < \nicefrac{1}{\sqrt{\d(A)}}$
to invert
\begin{equation*}
    (\1 + R^* B R)^{-1} = \sum_{n = 0}^{\infty} (- R^* B R)^n,
\end{equation*}
in $(\L^p)^N$. The inverse is compatible since the same Neumann series converges also in $(\L^2)^N$ owing to $c(2)=1$. Next, since $(- \Delta)^{\nicefrac{1}{2}} \colon (\Wdot^{s,p})^N \to (\Wdot^{s-1,p})^N$ is an isomorphism for $s = 0,1$ and all $p \in (1, \infty)$, it follows that also $- \Div A \nabla \colon (\Wdot^{1,p})^N \to (\Wdot^{-1,p})^N$ is invertible with compatible inverse.
\end{proof}

The constant $c(p)$ has a long history and that $c(p)$ is controlled from above by a dimensionless constant goes back to Stein \cite{Stein_d_independence_of_RT_norm}. Its exact value remains unknown to date. The best known estimates can be used to determine an improvement for $q_+(L) - 2$ explicitly. Dragi\v{c}evi\'{c} and Volberg have shown in \cite[Cor.~ 0.2]{Riesz-Trafo_Linear-bound} that
\begin{equation} \label{eq: Dimensionless improvement: Upper bound for c(p): Vol-Drag}
    c(p) \leq 2 (p -1) \qquad (p \geq 2).
\end{equation}
(Note that their short argument applies word-by-word to $\C^N$-valued functions.) However, this does not give $c(2) =1$, suggesting that their bound can be improved by interpolation for $p > 2$ not too far away from $2$. In fact, this is the case we are most interested in and we include the proof of the following elementary lemma. A similar argument for the Ahlfors--Beurling transform is found in \cite{Ahlfors-Beurling-Transform}.

\begin{lemma}\label{Dimensionless improvements for the critical numbers: Lemma: Upper bound for c(p)}
If $\sigma \approx 5.69061$ is the unique real solution to
\begin{equation*}
    \ln (2 \sigma -2 ) = \frac{\sigma(\sigma-2)}{2(\sigma-1)},
\end{equation*}
and $2 \leq p \leq \sigma$, then
\begin{equation}  \label{eq: Dimensionless improvement: Upper bound for c(p)}
    c(p) \leq \left( \e^{\frac{\sigma^2}{\sigma -1}} \right)^{\frac{1}{2} - \frac{1}{p}} \leq 2(p-1).
\end{equation}
\end{lemma}

\begin{proof}
Fix $q \geq p$. Riesz--Thorin interpolation and the fact that $c(2) =1$ yields
\begin{equation*}
    c(p) \leq c(q)^{\theta} \quad \text{with} \quad \frac{1}{p} = \frac{1- \theta }{2} + \frac{\theta}{q}.
\end{equation*}
We insert this value of $\theta$ and use \eqref{eq: Dimensionless improvement: Upper bound for c(p): Vol-Drag} to get
\begin{equation*}
    c(p) \leq \left( (2(q-1))^{\frac{2q}{q-2}} \right)^{\frac{1}{2} - \frac{1}{p}} = \left( \e^{\frac{2q}{q-2} \ln (2q -2)} \right)^{\frac{1}{2} - \frac{1}{p}}.
\end{equation*}
Since $\nicefrac{1}{2} - \nicefrac{1}{p} > 0$, we have to minimize the expression $\frac{2q}{q-2} \ln (2q -2)$ over $q \geq p$.  A straightforward calculation shows that the minimum over $q \geq 2$ is  attained in $q=\sigma$. Since $p \leq \sigma$, this must also be the global minimum over $q \geq p$ and the first estimate in \eqref{eq: Dimensionless improvement: Upper bound for c(p)} follows. The second estimate in  \eqref{eq: Dimensionless improvement: Upper bound for c(p)} follows simply because $q=p$ cannot give a better bound.
\end{proof}

\begin{proof}[\rm \bf{Proof of Theorem~\ref{Introduction: Theorem: q+ improvement}}]
Let
\begin{align*}
	\Phi \colon [2, \infty) \to [1, \infty), \quad \Phi(p) \coloneqq \begin{cases}
		\left( \e^{\frac{\sigma^2}{\sigma -1}} \right)^{\frac{1}{2} - \frac{1}{p}} \quad &(p \leq \sigma), \\ 2(p-1) \quad &(p > \sigma).
	\end{cases}
\end{align*}
Note that $\Phi$ is continuous, bijective and $c(p) \leq \Phi(p)$ for all $p \geq 2$, see \eqref{eq: Dimensionless improvement: Upper bound for c(p): Vol-Drag} and \eqref{eq: Dimensionless improvement: Upper bound for c(p)}. Hence, Proposition~\ref{Dimensionless improvements for the critical numbers: Proposition: Memoirs extrapolation} yields
\begin{align*}
	q_+(L) \geq \sup \Big\{ p >2 : \Phi(p) < \nicefrac{1}{\sqrt{\d(A)}} \Big\} = \Phi^{-1} \left(\nicefrac{1}{\sqrt{\d(A)}} \right).
\end{align*}
Inverting $\Phi$ leads to
\begin{equation*}
	q_+(L) \geq
	\begin{cases}
		\frac{2}{1 + \frac{\sigma -1}{\sigma^2} \ln(\d(A))} &\quad \text{if} \quad \frac{1}{4 (\sigma -1)^2} \leq \d(A), \\
		\frac{1}{2 \sqrt{\d(A)}} +1  &\quad \text{if} \quad \d(A) \leq \frac{1}{4 (\sigma -1)^2},
	\end{cases}
\end{equation*}
which proves the theorem.
\end{proof}


\end{document}